\documentclass[a4paper,10pt]{article}
\usepackage{amssymb,amsmath,amsthm,latexsym}
\usepackage{color}
\usepackage{textcomp}
\usepackage[active]{srcltx}
\usepackage{colortbl}

\usepackage{color}
\usepackage{mathrsfs}
\usepackage{graphicx}
\usepackage{lineno}

\usepackage[multiple]{footmisc}

\usepackage{fancyhdr}

\newtheorem{thm}{Theorem}[section]
\newtheorem{cor}[thm]{Corollary}
\newtheorem{lem}[thm]{Lemma}
\newtheorem{prop}[thm]{Proposition}
\theoremstyle{definition}

\theoremstyle{remark}
\newtheorem{rem}[thm]{Remark}
\numberwithin{equation}{section}

\title{On the  relations between the zeros of a polynomial and its Mahler measure}

\author{M. Ounaies, G. Rhin
and J.-M. Sac-\'Ep\'ee\footnote{M. Ounaies, IRMA, France, myriam.ounaies@math.unistra.fr\newline
G. Rhin,
 J.-M. Sac-\'Ep\'ee, IECL,
 France, georges.rhin@univ-lorraine.fr, jean-marc.sac-epee@univ-lorraine.fr}}

\begin{document}

\maketitle
\newcommand{\D}{\mathbb{D}}
\newcommand{\C}{\mathbb{C}}
\newcommand{\N}{\mathbb{N}}
\newcommand{\R}{\mathbb{R}}
\newcommand{\Z}{\mathbb{Z}}
\newcommand{\dist}{\operatorname{dist}}

\renewcommand{\qedsymbol}{$\blacksquare$}

\makeatletter
\def\blfootnote{\xdef\@thefnmark{}\@footnotetext}
\makeatother

\begin{abstract}
In this work, we are dealing with some properties relating the zeros of a polynomial and its Mahler measure. We provide estimates on  the number of real zeros of a polynomial, lower bounds on the distance between the zeros of a polynomial and non-zeros  located on the unit circle  and a lower bound  on the number of zeros of a polynomial  in the disk $\{\vert x-1\vert<1\}$.   \end{abstract}


\section{Introduction} 
 For a polynomial $P\in \C[x]$ of the form
 \[P(x)=\displaystyle\sum_{j=0}^d a_j x^j= \displaystyle a_d \prod_{k=1}^d (x-\mu_k), a_d\not=0,\] 
its Mahler measure \cite{Mahler 62}  is given by 

 \[M(P) = \exp\Big(\displaystyle\frac{1}{2\pi}\displaystyle\int_{0}^{2\pi} \log |P(e^{i t})| dt\Big).\]
  Using the Jensen's formula \cite{Jensen}, it  can be expressed as 
  \[M(P) = |a_d|\displaystyle\prod_{|\mu_k|\geq 1}|\mu_k|.\]

 \noindent For an algebraic number $\alpha$, its Mahler measure $M(\alpha)$ is that of its minimal polynomial in $\Z[x]$. We refer the reader to \cite{Smyth 08} for a survey of results on Mahler measure of algebraic numbers. 
  
 \noindent When $P\in \Z[x]$, clearly $M(P)\ge 1$ and by Kronecker's theorem \cite{Kronecker 57}, the only irreducible monic integer polynomials $P$ such that  $M(P)=1$, apart from the monomial $x$,  are cyclotomic polynomials. In other words, the only algebraic numbers $\alpha$ such that $M(\alpha)=1$ are $0$ and the roots of unity.

In  a paper from 1933 \cite{Lehmer 33}, Lehmer exhibited the polynomial  $P_L(x) = x^{10}+x^9-x^7-x^6-x^5-x^4-x^3+x+1$ whose Mahler measure $M(P_L)=1.176280\dots$ and raised the question whether there exists $P\in \Z[x]$ such that $1<M(P)<M(P_L)$.

\noindent This question is still open and the existence of a constant $\epsilon>0$ such that all $P\in \Z[x]$ satisfy either  $M(P)=1$ or $M(P)>1+\epsilon$ is now known as Lehmer's conjecture. 

For irreducible noncyclotomic polynomials $P\in \Z[x]$ of degree $d$,  E. Dobrowolski  \cite{Dobrowolski 79} showed a lower bound of the form   \[M(P)> 1+C \left(\frac{\log \log d}{\log d}\right)^3,\]
and it is the best unconditional lower bound for $M(P)$ known thus far, up to the constant $C$.

\noindent A polynomial $P$ of degree $d$ is said to be  self-reciprocal if $\displaystyle P(x)=x^dP\left(\frac{1}{x}\right)$. C. Smyth \cite{Smyth 71} proved an  important theorem in the  positive direction of  Lehmer's conjecture:
\begin{thm}\label{smyth}[Smyth]
If $P\in \Z[x]$ is not self-reciprocal and satisfies $P(0)P(1)\not=0$, then $M(P)\ge \theta_0$, where  $\theta_0= 1. 324717\dots$ is the real root of the polynomial $x^3-x-1$.
\end{thm}

\noindent Let us also cite a theorem due to A. Schinzel \cite{Schinzel 73}:
\begin{thm}\label{schinzel}[Schinzel]
If  $P$ is a monic  integer polynomial of degree $d$ whose all roots are real and such that $P(0)P(1)P(-1)\not=0$, then 
 \[M(P)\ge \left(\frac{1+\sqrt 5}{2}\right)^{\frac{d}{2}}.\] 
 \end{thm}

 We will denote by $E_{\theta}$ the set of monic irreducible polynomials  $P\in \Z[x]$ such that $1<M(P) <\theta\le \theta_0$, where $\theta_0$ is as in in Theorem \ref{smyth}.   Many authors have focused on the computational search for such polynomials. D. Boyd \cite{Boyd 80, Boyd 89} gave a complete list of  polynomials in $E_{1.3}$ up to degree 20. This list was extended by Mossinghoff \cite{Mossinghoff 98} up to degree 24, then by Flammang, Rhin and Sac-\'Ep\'ee  \cite{Flammang et al. 06} with  exhaustive results up to degree 40 and by Mossinghoff, Rhin and Wu \cite{Mossinghoff et al. 08} who provided an exhaustive list up to degree 44.

M. Mossinghoff \cite{Mossinghoff 07} maintains a website that collects all known polynomials in $E_{1.3}$  of degree less than 180. New polynomials have been added to the list available on this website, from the works  of Rhin and Sac-\'Ep\'ee \cite{Rhin and SE  03}, Mossinghoff, Pinner and Vaaler \cite{Mossinghoff et al. 98}, and El Otmani, Maul, Rhin and Sac-\'Ep\'ee \cite{El Otmani Maul Rhin Sac-Epee 17}.

 In the present work, we are dealing with estimates about the zeros of a polynomial with complex coefficients  in relation with its Mahler measure.

 Because of their importance for Lehmer's conjecture, we are interested in applying  some of our results from Sections 4, 5 and 6 to obtain information on the zero sets of polynomials in $E_{\theta}$.
 
The paper is organised as follows: 
Section 2 provides notations, preliminary results and first properties of $E_\theta$.

Section 3 deals with  lower bounds on $\vert \mu-\omega\vert$ where  $P(\mu)=0$, $\vert \omega\vert=1$, $P(\omega)\not=0$.  Subsection 3.1 focusses on integer polynomials while Subsection 3.2 deals with real or complex polynomials, mainly self-reciprocal. The lower bounds involve either $M(P)$ or $L(P)$. 

A. Dubickas \cite{Dubickas 95}  showed that for all $\varepsilon>0$, there exists $D_0(\varepsilon)$ such that for all monic, irreducible, non-cyclotomic polynomials $P\in \Z[x]$ of degree $d\ge D_0(\varepsilon)$  and for all $\mu\in P^{-1}(0)$,
  \begin{equation}\label{DubIntro}  \vert\mu-1\vert > \exp\left(-\left(\frac{\pi}{4}+\varepsilon\right)\sqrt{d\log d\log M(P)}\right).
  \end{equation}
\noindent  Our Proposition \ref{Dubickas}  improves on (\ref {DubIntro}) by substituting the constant $\frac{\pi}{8}$ to  $\frac{\pi}{4}$ in the special case where $P$ is self-reciprocal. 

\noindent  For polynomials $P\in \C[x]$ of degree $d$, Theorem \ref{general} provides the following inequality   for all $\mu\in P^{-1}(0)$:
\[\vert  \mu-1\vert \ge \frac{\vert P(1)\vert}{edL(P)}.\]
In particular, if all the coefficients of $P$ are positive, then 
\[\vert  \mu-1\vert \ge  \frac{1}{ed}.\]

 In Section 4, we  generalise Theorem \ref{schinzel}   to  polynomials in $\C[x]$ and we strengthen the dependence on $P(0)$, $P(-1)$ and $P(1)$. Our Theorem \ref{Schinzel} states that for  monic polynomials $P\in \C[x]$  of degree $d$   having $m\ge 1$  real zeros and such that  $P(0)P(1)P(-1)\not=0$, the following holds:
\begin{equation*}
M(P)\ge \left(\frac{\vert P(1)P(-1)\vert^{\frac{1}{m}}+\left(4^{\frac{d}{m}}\vert P(0)\vert^{\frac{2}{m}}+\vert P(1)P(-1)\vert^{\frac{2}{m}}\right)^{\frac{1}{2}}}{2^{\frac{d}{m}}}\right)^{\frac{m}{2}}.
\end{equation*}

 In Section 5, we give upper bounds for the number of  real zeros of  $\displaystyle P(x)=\sum_{j=1}^d a_j x^j \in \C[x]$ by terms involving (besides  its degree) its Mahler measure or its length $\displaystyle L(P)=\sum_{j=1}^d \vert a_j\vert$. 
 
\noindent  A. Dubickas \cite{Dubickas 95} showed that for polynomials  $P\in \Z[x]$  of degree $d$ satisfying  $P(0)P(-1)P(1)\not=0$,  there is a constant $\alpha >0$ such that 
\begin{equation}\label{Dubickasintro}%
\vert P^{-1}(0)\cap \R\vert\le  \alpha \sqrt{d\log d \log M(P)}.
\end{equation}
Corollary \ref{L(P)Z} allows us  to improve (\ref{Dubickasintro}) by eliminating the term $\sqrt{\log d}$ in the special case where  $\log L(P)\le \sqrt{d\log M(P)}$.

 In Section 6, we focus on integer polynomials. The main result is the following  lower bound for the Mahler measure of a monic irreducible polynomial  $P\in \Z[x]$ of degree $d\ge D_0(\varepsilon)$  and satisfying $\log M(P)=o(\frac{d}{\log d})$:
\begin{equation*}
\log M(P) \ge  \left(C-\varepsilon\right)\frac{d}{K(P)^2\log d }\ \text{ with } C=\frac{4}{\pi^2}\log^2\left(\frac{1+\sqrt 5}{2}\right)
\end{equation*}
 and where we have denoted by $K(P)=\min\left(\left\vert P^{-1}(0)\cap\{\vert x-1\vert<1\}\right\vert, \left\vert P^{-1}(0)\cap\{\vert x+1\vert<1\}\right\vert \right)$.


\section{Notations and preliminary results } 

\noindent Consider a polynomial $P(x)=\displaystyle\sum_{j=0}^d a_j x^j\in \C[x]$ of degree $d$.  In our estimates in relation with $P$, besides $M(P)$, we will use the following quantities:
 \begin{equation*}\label{quant}
 \begin{split}
 &H(P)=\max_{0\le j\le d}\vert a_j\vert,\ \ \ L(P)=\sum_{j=0}^d\vert a_j\vert, \\
& \Vert P\Vert=\sup_{\vert x\vert=1} \vert P(x)\vert,\ \ \  L_2(P)=\left(\sum_{j=0}^d\vert a_j\vert^2\right)^{\frac{1}{2}}.
 \end{split}
 \end{equation*}
 The height $H(P)$  and the length $L(P)$ were introduced in \cite{Mahler 62}  by K. Mahler who compared them to $M(P)$,  showing the inequalities
  \begin{equation}\label{ineq}
 \begin{split}
&\vert a_j\vert \le M(P) {d\choose j}\ \ (j=0,\ldots, d),\\
&M(P)\le L_2(P)\le L(P)\le 2^dM(P),\\
&H(P)\le 2^{d-1}M(P),\ \ \ M(P)\le (d+1)^{\frac{1}{2}} H(P).\\
\end{split}
\end{equation}
Using Cauchy-Schwarz inequality, we also have
\begin{equation}\label{LP}
\vert P(1)\vert\le\Vert P\Vert\le L(P)\le (d+1)^{\frac{1}{2}} L_2(P)
\end{equation}
and by  Parseval formula, 
\begin{equation}\label{Parseval}
L_2(P)=\left(\frac{1}{2\pi}\int_0^{2\pi}\vert P(e^{i\theta})\vert^2d\theta\right)^{\frac{1}{2}}\le \Vert P\Vert.
\end{equation}
\noindent Define the reciprocal of $P$ by
\[P^*(x)=x^dP\left(\frac{1}{x}\right)=\sum_{j=0}^d a_{d-j} x^j.\]
\noindent  Clearly, $P$ is self-reciprocal iff $P=P^*$.

\noindent Let $P(x)=\displaystyle\sum_{j=0}^d a_j x^j=\displaystyle\prod_{k=1}^d(x-\mu_k)\in \C[x]$, $a_d=1$, $a_0\not=0$. 
 We will denote by 
\begin{equation}\label{Q}
 Q(x)=a_0^{-1}P(x)P^*(x)=\prod_{k=1}^d (x-\mu_k)\prod_{k=1}^d (x-\mu_k^{-1})=\sum_{j=0}^{2d}A_j x^j.
\end{equation}
Then $Q$ is monic and self-reciprocal and
\begin{equation}\label{Q-P}
M(Q)=\vert a_0\vert^{-1}M(P)^2,\ Q(1)=a_0^{-1}P(1)^2,\ Q(-1)=(-1)^da_0^{-1}P(-1)^2.
\end{equation}
Besides,  the coefficients are given by
\[A_j=a_0^{-1}\sum_{k=0}^ja_k a_{d-j+k}\ \  (j=0,\ldots,d)\]
and they satisfy
\begin{equation}\label{A-P}
\sum_{j=0}^d\vert A_j\vert \le \vert a_0^{-1}\vert\sum_{j=0}^d\sum_{k=0}^{j}\vert a_k \vert \vert a_{d-j+k}\vert= \vert a_0^{-1}\vert\sum_{k=0}^d\left(\vert a_k\vert\sum_{j=k}^{d}\vert a_j\vert \right)\le \vert a_0^{-1}\vert L(P)^2.
\end{equation}

 For a  set $A\subset \C$, we will denote by $\vert P^{-1}(0)\cap A\vert$ the number of zeros of $P$ belonging to $A$ counted with multiplicities. 
\begin{prop}\label{b_r} Let $r>1$ and $P$ be a monic polynomial in $\C[x]$.  Then
\[\vert P^{-1}(0)\cap \{\vert x\vert> r\} \vert <\frac{\log{M(P)}}{\log r}.\]
\end{prop}
\begin{proof}
Put $N=\vert P^{-1}(0)\cap \{\vert x\vert> r\}\vert$. Then
\[r^N<  \prod_{\mu\in  P^{-1}(0)\cap \{\vert x\vert> r\}}\vert\mu\vert\le \prod_{\mu \in  P^{-1}(0)\cap \{\vert x\vert\ge 1\}}\vert\mu\vert=M(P).\]
\end{proof}

 \noindent For $1<\theta< \theta_0$, where $\theta_0$ is given by Theorem \ref{smyth}, let 
$E_{\theta}$ be the set of integer monic  irreducible polynomials $P$ such that $1<M(P)\le \theta$ and satisfying conditions  c1 and c2:
\begin{itemize}
\item[c1.] $P$ is primitive, i.e.,  it cannot be written as $Q(x^m)$ with $m\ge 2$, $Q\in \Z[x]$,
\item[c2.]  $a_{j_0}> 0$, where  $j_0=\min\{j\ge 1 \text{ such that } a_j\not=0\}$.  This condition is satisfied  either by $P(x)$ or  by $P(-x)$.   
\end{itemize}
Conditions c1 and c2 are usually added by the authors looking for polynomials with small Mahler measure via computational methods in order to avoid redundancy, since
\[M(P(x))=M(P(-x)) \text{ and } M(P(x^m))=M(P(x)).\]
\begin{rem}
\begin{itemize}
\item For any $P\in E_\theta$ and $\mu \in P^{-1}(0)$, $P$ is the minimal polynomial of $\mu$ in $\Z[x]$.
\item By Theorem \ref{smyth}, all polynomials of $E_{\theta}$ are self-reciprocal.
\item Any self-reciprocal polynomial of odd degree is divisible by $(x+1)$. Thus all polynomials in $E_\theta$ are of even degree. 
\item By Kronecker's theorem, the cyclotomic polynomials are excluded of $E_\theta$ because their Mahler measure is equal to $1$. 
\end{itemize}
\end{rem}
The first properties satisfied by the zero set $\displaystyle Z_\theta=\bigcup_{P\in E_\theta} P^{-1}(0)$ are listed in the proposition below.

\begin{prop}${}$

\begin{enumerate}
\item All $P^{-1}(0)$, $P\in E_\theta$ are disjoint and they contain only  simple zeros.
\item For all $P\in E_\theta$, $P^{-1}(0)$ is symmetric with regard to $ \{\vert x\vert=1\}$ and $\{\Im x=0\}$.
\item $ Z_\theta\subset \{\frac{1}{\theta}\le \vert x\vert \le\theta\}$.
\item $Z_\theta \cap \{\Im x\not=0\}\subset \{\frac{1}{\sqrt{\theta}}\le \vert x\vert \le \sqrt{\theta}\}$.
\item For all $P\in E_\theta$ of degree $2n$, and for all $r>1$,  \[\vert P^{-1}(0)\cap \{\frac{1}{r}\le \vert x\vert\le r\}\vert > 2(n-\frac{\log\theta}{\log r}).\]
\item $Z_\theta$ does not contain any root of unity.
\item $Z_\theta$ does not intersect the imaginary axe.

\end{enumerate}
\end{prop}
\begin{proof}
\begin{enumerate}
\item This follows immediately after  the uniqueness of the minimal polynomial of an algebraic integer. 
\item This is because $P(x)=0 \Leftrightarrow P(\bar x)=0 \Leftrightarrow  P(1/x)=0$.
\item Let $P\in E_\theta$.  By symmetry of $P^{-1}(0)$ with regards to $ \{\vert x\vert=1\}$, it suffices to note that if $x\in P^{-1}(0)\cap \{\vert x\vert>1\}$ then $\vert x\vert\le M(P)\le \theta$.  
\item  Let $P\in E_\theta$. By symmetry of $P^{-1}(0)$ with regards to $ \{\vert x\vert=1\}$ and $\{\Im x=0\}$, it suffices to note that if $x\in P^{-1}(0)\cap \{\vert x\vert>1\}\cap \{\Im x\not=0\}$ then $\vert x\vert^2\le M(P)\le \theta$. 
\item This is a consequence of Proposition \ref{b_r} and 2.
\item The minimal polynomials corresponding to the roots of unity are cyclotomic polynomials, which are excluded from $E_\theta$.
\item Assume that for some $y\in \R$ and $\displaystyle P=\sum_{j=0}^{2n}a_j x^j\in E_{\theta}$ $(a_{2n}\not=0)$, $P(iy)=0$. Then $\displaystyle i\Im(P(iy))=\sum_{k=0}^{n-1} a_{2k+1} (iy)^{2k+1}=0$. But at least one of the coefficients $a_{2k+1}$ is not zero by condition c1.  This would contradict the fact that $P$ is the minimal polynomial of $iy$.

\end{enumerate}
\end{proof}
For illustration purposes, the graph on the left in Figure 1 shows all known roots belonging to $Z_{1.3}$, corresponding to polynomials of degree less than 180. In particular, one can visualize the properties of symmetry of this set. The small Mahler measure polynomials from which these roots are derived are listed on Michael Mossinghoff's website \cite{Mossinghoff 07}.  
\begin{figure}
    \caption{\label{etiquette} The known elements of $Z_{1.3}$ obtained from polynomials of degree less than 180}
    \includegraphics[scale=0.25]{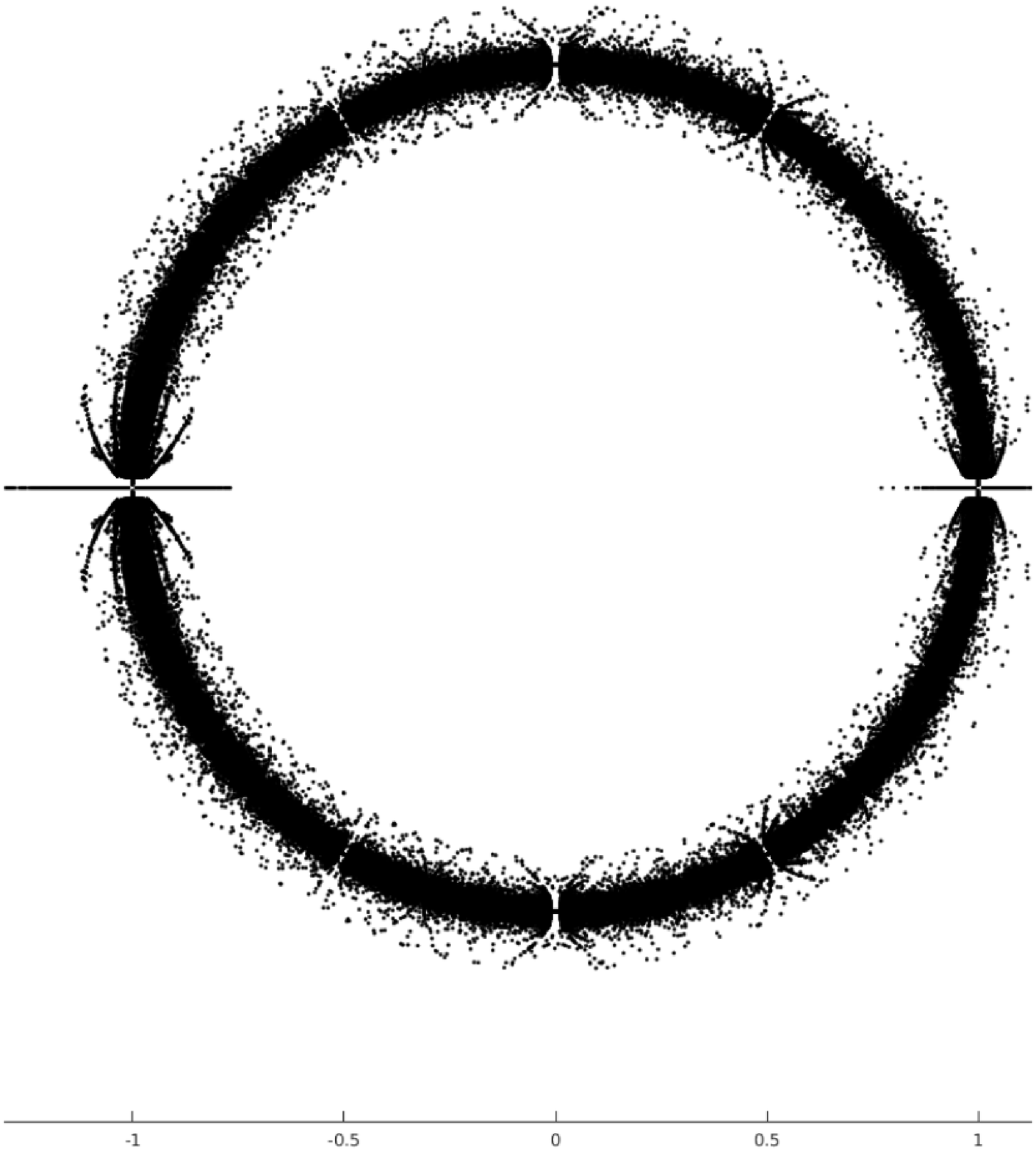}\hfill
    \includegraphics[scale=0.25]{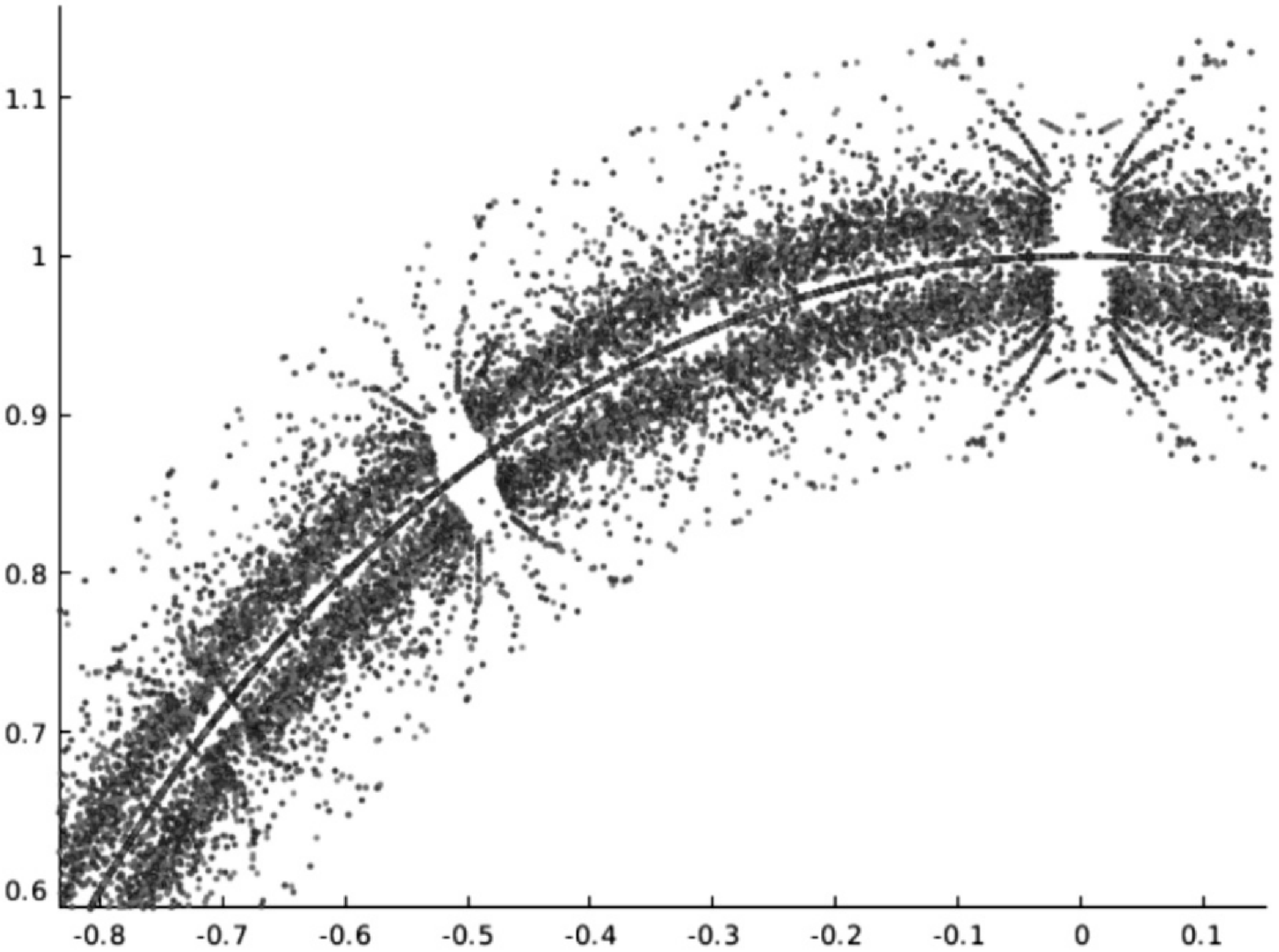}
\end{figure}


\section{Distance between zeros of a polynomial and non-zeros of modulus one}
On Figure 1, we notice that the roots of the polynomials in $E_{1.3 }$ seem to stay far from the roots of unity. The results given in this section, when specialised to polynomials in $E_{\theta}$, $\theta\le \theta_0$, allow us to estimate from below the distance between the zeros of $P\in E_\theta$ and the roots of unity.  Most of our  results deal with  polynomials which are self-reciprocal of even degree like the ones in  $E_{\theta}$.

\subsection{Lower bounds involving the Mahler measure}
For an algebraic number $\mu\not=1$ of degree $d$, Liouville inequality  gives:
 \[\vert\mu-1\vert \ge 2^{1-d}M(\mu)^{-1}.\]
We prove the following Liouville-type inequalities: 
\begin{prop}\label{lower}  Assume that $P\in \Z[x]$ is monic, self-reciprocal of degree $2n$ and does  not vanish at the roots of unity. If  $P(\mu)=0$ and
 $\omega^m=1$ or $\omega^m=-1$ $(m\in \N^*)$, then
\[\text{ If } \mu \in \R \text { or } \vert \mu\vert=1, \ \ \vert  \mu-\omega\vert\ge  m^{-1}2^{1-n}M(P)^{-\frac{m}{2}},\]
\[\text{ If } \mu \notin \R \text{ and }\ \vert \mu\vert \not=1,\ \  \vert  \mu-\omega\vert\ge m^{-1}2^{1-\frac{n}{2}}M(P)^{-\frac{m}{4}},\]
\end{prop}
\begin{proof}
By substituting $P(-x)$ to $P(x)$,  we may assume that $\omega^m=1$ and  by substituting $\bar \omega$ to $\omega$, we may assume that $\Im \mu \ge 0$.  We divide $P^{-1}(0)\cap\{\vert x\vert\ge 1, \ \Im x\ge 0\}$ into two sets:
\begin{equation*}
\begin{split}
&\mu \in A \Leftrightarrow  P(\mu)=0,\ \vert \mu\vert=1, \Im \mu>0 \text{ or }  \vert \mu\vert>1, \Im \mu=0,\\
&\mu\in B \Leftrightarrow P(\mu)=0,\ \vert \mu\vert>1, \Im \mu >0.
\end{split}
\end{equation*}
Then $P^{-1}(0)=\{\mu, \mu^{-1},\ \mu \in A\}\cup\{\mu, \mu^{-1}, \bar \mu, \bar \mu^{-1},\ \mu \in B\}$ and
\begin{equation*}\label{Liouville}
 \prod_ {\mu\in  P^{-1}(0)}\vert 1-\mu^m\vert=M(P)^m\prod_{\mu\in A}\vert 1-\mu^{-m}\vert^{2} \prod_{\mu\in B}\vert 1-\mu^{-m}\vert^{4}. \\
\end{equation*}
 $\displaystyle\prod_{\mu\in  P^{-1}(0)}(1-\mu^m)$ is a symmetric polynomial in the zeros of $P$ thus it is an integer and as $P$ does not vanish at the roots of unity, it is  a non-zero integer. We deduce that 
 \begin{equation*}
\begin{split}
1\le M(P)^{\frac{m}{2}}\prod_{\mu\in A}\vert 1-\mu^{-m}\vert \prod_{\mu\in B}\vert 1-\mu^{-m}\vert^{2} .  \\
\end{split}
\end{equation*}
For all $\mu \in A\cup B$, we have $\vert1-\mu^{-m}\vert \le 2$ thus
\[\forall \mu \in A, \ \ 1\le M(P)^{\frac{m}{2}} 2^{n-1}\vert 1-\mu^{-m}\vert,\]
\[\forall \mu \in B, \ \ 1\le M(P)^{\frac{m}{2}} 2^{n-2}\vert 1-\mu^{-m}\vert^{2}.\]
To obtain the desired estimate for all  $\mu \in A\cup B$, we use the following
\begin{equation*}\label{mum}
\vert 1-\mu^{-m}\vert \le \vert \bar\omega-\mu^{-1}\vert \sum_{j=0}^{m-1}\vert \mu^{-j}\vert\le m \vert \bar\omega-\mu^{-1}\vert =m\vert \mu\vert^{-1}\vert \omega-\mu\vert\le  m\vert \omega-\mu\vert.
\end{equation*}
If $\vert \mu\vert<1$, we apply the former to  $\bar \mu^{-1}$. and we use  $\vert 1-\bar\mu^{m}\vert \le m \vert \omega-\mu\vert$.
\end{proof}

For integer polynomials of small Mahler measure,  there are better lower bounds due to Mignotte and Waldschmitt \cite{Mignotte and Waldschmidt 94}. They showed that for all $\varepsilon>0$, there exists $D_0(\varepsilon)$ such that for $d\ge D_0(\varepsilon)$
 \[\vert\mu-1\vert > \exp\left(-(1+\varepsilon)\sqrt{d\log d\log M}\right).\] 
This was improved by Dubickas \cite{Dubickas 95}  as follows:
  \begin{equation}\label{Dub}  \vert\mu-1\vert > \exp\left(-\left(\frac{\pi}{4}+\varepsilon\right)\sqrt{d\log d\log M}\right).
  \end{equation}

With a minor modification in Dubickas proof  of Theorem 1 in \cite{Dubickas 95}, we obtain a better result than (\ref{Dub}) in the special case when $P$ is self-reciprocal. This is displayed in the next proposition.  
\begin{prop}\label{Dubickas}  Assume that $P\in \Z[x]$ is  monic, irreducible, self-reciprocal  of degree $d$ and satisfies $M(P)>1$. Let  $P(\mu)=0$  and
 $\omega^m=1$ or $\omega^m=-1$ $(m\in \N^*)$. Then for all $\varepsilon>0$, there exists $D_0(\varepsilon)$ such that for $d\ge D_0(\varepsilon)$
\[\vert \omega- \mu\vert\ge  \exp\left(-\left(\frac{\pi \sqrt m}{8}+\varepsilon\right)\sqrt{d\log d\log M(P)}\right).\]
If $\mu \notin \R$ and $\vert \mu\vert \not=1$, then
\[\vert \omega- \mu\vert\ge    \exp\left(-\left(\frac{\pi \sqrt m}{16}+\varepsilon\right)\sqrt{d\log d\log M(P)}\right).\]
\end{prop}
\subsection{Lower bounds involving the coefficients} 
\begin{thm}\label{general}
 Let $P\in \C[x]$ be of degree $d\in \N^*$. Let $\mu$ be a zero of $P$ with multiplicity $m_\mu$  and $\vert\omega\vert=1$. Then, 
\[\vert  \mu-\omega\vert \ge d^{-1}\left(e^{-1}\frac{\vert P(\omega)\vert}{L(P)}\right)^{\frac{1}{m_\mu}}.\]
In particular, if all the coefficients of $P$ are positive, then 
\[\vert  \mu-1\vert \ge d^{-1}e^{-\frac{1}{m_\mu}}.\]
If $\vert P(\omega)\vert=\Vert P\Vert$, then 
\[\vert  \mu-\omega\vert \ge d^{-1}\left(e^{-1}(d+1)^{-\frac{1}{2}}\right)^{\frac{1}{m_\mu}}.\]
\end{thm}

\begin{proof}
We may assume $P(\omega)\not=0$.
For all $\theta\in [0,2\pi]$, 
\[\vert P(\omega+d^{-1} e^{i\theta})\vert \le  (1+d^{-1})^{d}L(P)\le eL(P).\] Let $\mu_k$  be the distinct zeros of $P$ with respective multiplicities $m_k$. Applying Jensen's formula to $\frac{P(x)}{P(\omega)}$ in the disk $\{\vert x-\omega\vert<d^{-1}\}$, we find:
\begin{equation*}
\sum_{\vert \mu_k-\omega\vert \le d^{-1}}m_k \log \frac{d^{-1}}{\vert \mu_k-\omega \vert }=\frac{1}{2\pi} \int_0^{2\pi} \log \left \vert  \frac{P(\omega+ d^{-1} e^{i\theta})}{P(\omega)} \right\vert d\theta \le \log\frac{
e L(P)}{\vert P(\omega)\vert}.
\end{equation*}
For each $k$, either $\vert \mu_k-\omega\vert \ge d^{-1}\ge d^{-1} \frac{\vert P(\omega)\vert}{eL(P)}$ or, using the estimate above, 
\[\left(d\vert  \mu_k-\omega\vert\right)^{m_k} \ge \frac{\vert P(\omega)\vert}{eL(P)}.\]
In the case where all the coefficients of $P$ are positive, $P(1)=L(P)$.

\noindent In the case where $P(\omega)=\Vert P\Vert$, we use the inequalities (\ref{LP}) and (\ref{Parseval}):
\[L(P)\le (d+1)^{\frac{1}{2}}\Vert P\Vert=(d+1)^{\frac{1}{2}}\vert P(\omega)\vert.\]
\end{proof}

The rest of this subsection deals with self-reciprocal polynomials. As  any self-reciprocal polynomial  of odd degree is the product of $(x+1)$ and a self-reciprocal polynomial  of even degree, we will only consider polynomials of degree $d=2n$.   
 \begin{lem}\label{Jensen}
Let $P(x)=\displaystyle\sum_{j=0}^{2n}a_j x^j$ be a self-reciprocal polynomial in $\C[x]$ of degree $2n$.  Let $\vert \omega\vert=1$, $P(\omega)\not=0$ and $\rho\in (0,1)$. Then
  \begin{equation}\label{rem}
  \begin{split}
\sum_{\mu\in P^{-1}(0),\vert\mu-\omega\vert\le \rho}\log\frac{\rho}{\vert\mu-\omega\vert}
\le \log\left(1+\frac{\rho(1+(1-\rho)^{-n})}{\vert P(\omega)\vert} \sum_{j=0}^{n-1} (n-j)\vert a_j\vert\right).
\end{split}
\end{equation}
If $\omega\in \{-1,1\}$, 
\begin{equation}\label{rem1}
\begin{split}
\sum_{\mu\in P^{-1}(0),\vert\mu-\omega\vert\le \rho}\log\frac{\rho}{\vert\mu-\omega\vert}
\le \log\left(1+\frac{\rho^2(1-\rho)^{-n}}{\vert P(\omega)\vert} \sum_{j=0}^{n-1} (n-j)^2\vert a_j\vert\right).
\end{split}
\end{equation}
\end{lem}

\begin{proof}
Jensen's formula applied  to the function $\frac{x^{-n}P(x)}{P(\omega)}$ in the disk $\{\vert x-\omega\vert<\rho\}$ gives
\begin{equation}\label{Jensenpf}
\sum_{\mu\in P^{-1}(0),\vert \mu-\omega\vert \le \rho}\log \frac{\rho}{\vert \mu-\omega \vert }=\frac{1}{2\pi} \int_0^{2\pi} \log \left \vert  \frac{(\omega+\rho e^{i\theta})^{-n}P(\omega+\rho e^{i\theta})}{P(\omega)} \right\vert d\theta.
\end{equation}
Let $x$ be a non-zero complex number. A straightforward computation shows that
\begin{equation*}\label{omega1}
\begin{split}
x^{-n}P(x)&=\bar\omega^nP(\omega )-\sum_{j=0}^{n-1}a_j {\bar \omega}^{n-j} (x^{n-j}-\omega^{n-j} )(x^{j-n}-\omega^{n-j}).
\end{split}
\end{equation*}
We use 
\begin{equation*}
\begin{split}
\left\vert(x^{n-j}-\omega^{n-j} )(x^{j-n}-\omega^{n-j})\right\vert&=\vert x-\omega\vert\vert x^{j-n}-\omega^{n-j}\vert\left\vert \sum_{k=0}^{n-j-1}x^k\omega^{-1-k}\right\vert \\
&=\vert x-\omega\vert \left\vert \sum_{k=0}^{n-j-1}(x^{k+j-n}\omega^{-1-k}-x^k \omega^{n-j-1-k})\right\vert\\
&\le \vert x-\omega\vert (n-j)(1+\max(\vert x\vert^{-n},\vert x\vert^{n})
\end{split} 
\end{equation*}
and we readily obtain
\begin{equation}\label{omega1}
\left\vert \frac{x^{-n}P(x)}{P(\omega)}\right\vert \le 1+\frac{\vert x-\omega\vert  (1+\max(\vert x\vert^{-n},\vert x\vert^{n})}{\vert P(\omega)\vert} \sum_{j=0}^{n-1} (n-j)\vert a_j\vert.
\end{equation}
(\ref{rem}) follows from (\ref{Jensenpf}), (\ref{omega1}) and the fact that for all $\theta\in [0,2\pi]$, 
\begin{equation}\label{theta}
1-\rho\le \vert \omega+\rho e^{i\theta}\vert\le 1+\rho\le (1-\rho)^{-1}.
\end{equation}
 In  the case where $\omega\in \{1,-1\}$, we write
 \begin{equation*}
\begin{split}
\left\vert(x^{n-j}-\omega^{n-j} )(x^{j-n}-\omega^{n-j})\right\vert&=\vert x\vert^{-1} \vert x-\omega\vert^2 \left\vert \sum_{k=0}^{n-j-1}x^k\omega^{-1-k}\right\vert \left\vert\sum_{k=0}^{n-j-1}x^{-k}\omega^{-1-k}\right\vert\\
&\le \vert x-\omega\vert^2 (n-j)^2 \max(\vert x\vert^{-n},\vert x\vert^{n})
\end{split}
\end{equation*}
which gives the inequality
\begin{equation}\label{omega2}
\left\vert \frac{x^{-n}P(x)}{P(\omega)}\right\vert \le 1+\frac{\vert x-\omega\vert^2 \max(\vert x\vert^{n},\vert x\vert^{-n})}{\vert P(\omega)\vert}\sum_{j=0}^{n-1} (n-j)^2\vert a_j\vert.
\end{equation}
(\ref{rem1}) follows from (\ref{Jensenpf}), (\ref{theta}) and (\ref{omega2}).
\end{proof}

\begin{thm}\label{lower1}
Let $P(x)=\displaystyle\sum_{j=0}^{2n}a_jx^j$ be a self-reciprocal polynomial in $\R[x]$ of degree $2n$.   Assume $P(\mu)=0$, $\vert \omega\vert=1$, $P(\omega)\not=0$ and $\delta>1$. Then
\begin{equation}\label{alpha}
\begin{split}
 \frac{1}{\vert\mu-\omega\vert} \le \frac{n}{\log \delta}+1+\frac{1+\delta}{\vert P(\omega)\vert }\sum_{j=0}^{n-1}(n-j)\vert a_j\vert,\\
\end{split}
\end{equation}
Besides, if $\vert \mu\vert\not=1$ then
\begin{equation}\label{betagamma}
\frac{\vert \mu\vert}{\vert\mu-\omega\vert^2} \le  (\frac{n}{\log \delta}+1)^2+(\frac{n}{\log \delta}+1)\frac{1+\delta}{\vert P(\omega)\vert }\sum_{j=0}^{n-1}(n-j)\vert a_j\vert.
\end{equation}
 For $\omega\in \{-1,1\}$, 
\begin{equation}\label{alphabeta}
 \frac{\vert \mu\vert }{\vert\mu-\omega\vert^2} \le  (\frac{n}{\log \delta}+1)^2+\frac{\delta}{\vert P(\omega)\vert }\sum_{j=0}^{n-1}(n-j)^2\vert a_j\vert,
\end{equation}
Besides, if $ \mu \notin \R$   and $\vert \mu\vert\not=1$ then
\begin{equation}\label{gamma}
 \frac{\vert \mu\vert^2}{\vert\mu-\omega\vert^4} \le (\frac{n}{\log \delta}+1)^4+(\frac{n}{\log \delta}+1)^2\frac{\delta}{\vert P(\omega)\vert }\sum_{j=0}^{n-1}(n-j)^2\vert a_j\vert.
\end{equation}
\end{thm}

\begin{proof}
We apply Lemma \ref{Jensen} with  $\rho=\frac{\log \delta}{n+\log \delta}$. Let  $\mu \in P^{-1}(0)$ and $\vert \omega\vert=1$, $P(\omega)\not=0$. We may assume that $\vert \mu-\omega\vert\le \rho$, otherwise the inequalities are trivial. Then, applying  (\ref{rem}), 
 \begin{equation*}\label{alpha1}
  \begin{split}
 \frac{\rho}{\vert\mu-\omega\vert}&\le  \prod_{\underset{\vert\alpha-\omega\vert\le \rho}{\alpha \in P^{-1}(0)}}\frac{\rho}{\vert\alpha-\omega\vert} \le 1+\frac{\rho(1+(1-\rho)^{-n})}{\vert P(\omega)\vert} \sum_{j=0}^{n-1} (n-j)\vert a_j\vert. 
\end{split}
\end{equation*}
Now  we use the inequality $(1-\rho)^{-n}=\left(1+\frac{\log\delta}{n}\right)^n\le \delta$ and we obtain
(\ref{alpha}). 
\noindent In the rest of this proof, since $P(\bar\mu^{-1})=0$ and $\frac{\vert \mu\vert}{\vert \mu-\omega\vert^2}$ is invariant by substituting $\bar\mu^{-1}$ to $\mu$, we may assume that $\vert \mu\vert\ge 1$. 

\noindent If $\vert \mu\vert>1$, then $\bar\mu^{-1}\not=\mu$ and 
\begin{equation*}\label{mu}
\vert \bar\mu^{-1}-\omega\vert=\frac{\vert \mu-\omega\vert}{\vert \mu\vert}\le \vert\mu-\omega\vert\le \rho.
\end{equation*}
Thus
 \begin{equation*}\label{gamma1}
  \begin{split}
 \frac{\rho^2 \vert \mu\vert}{\vert \mu-\omega\vert^2}&\le  \prod_{\underset{\vert\alpha-\omega\vert\le \rho}{\alpha\in P^{-1}(0)}}\frac{\rho}{\vert\alpha-\omega\vert}\le 1+\frac{\rho(1+\delta)}{\vert P(\omega)\vert} \sum_{j=0}^{n-1} (n-j)\vert a_j\vert 
\end{split}
\end{equation*}
and this proves  (\ref{betagamma}). 

\noindent Let us assume now that $\omega\in \{-1,1\}$. 
Since $P(\mu^{-1})=0$ and $\vert  \mu^{-1}-\omega\vert \le \vert \mu-\omega\vert\le  \rho$, applying  (\ref{rem1}) we obtain
 \begin{equation*}\label{beta1}
  \begin{split}
\frac{\rho^2\vert \mu\vert}{\vert\mu-\omega\vert^2}\le   \prod_{\underset{\vert\alpha-\omega\vert\le \rho}{\alpha\in P^{-1}(0)}}\frac{\rho}{\vert\alpha-\omega\vert} \le 1+\frac{\rho^2\delta}{\vert P(\omega)\vert} \sum_{j=0}^{n-1} (n-j)^2\vert a_j\vert,
\end{split}
\end{equation*}
which proves (\ref{alphabeta}). 

\noindent Finally, to prove (\ref{gamma}), we note that if $\vert \mu\vert>1$ and  $\mu \notin \R$, then $\mu$, $\bar \mu$, $\mu^{-1}$ and $\bar\mu^{-1}$ are distinct zeros of $P$ and all  of them are located in the disk $\{ \vert x-\omega\vert \le \rho\}$. Thus 
\begin{equation*}\label{gamma2}
  \begin{split}
 \frac{\rho^4 \vert \mu\vert^2}{\vert\mu-\omega\vert^4}&\le  \prod_{\underset{\vert\alpha-\omega\vert\le \rho}{\alpha\in P^{-1}(0)}}\frac{\rho}{\vert\alpha-\omega\vert} \le 1+\frac{\rho^2\delta}{\vert P(\omega)\vert} \sum_{j=0}^{n-1} (n-j)^2\vert a_j\vert. 
\end{split}
\end{equation*}

\end{proof}

In the remainder of this section, we will use the following notation: for a polynomial $P$ of degree $2n$, we will write $ f(P)\gtrsim_{n\rightarrow \infty}  g(n)$
   if there exists an expression  $h(n)$  only depending on $n$ such that  $f(P)\ge h(n)$ and $\frac{h(n)}{g(n)}\rightarrow_{n\rightarrow \infty} 1$.

\begin{cor}\label{lower2}
Let $P(x)=\displaystyle\sum_{j=0}^{2n}a_jx^j$ be a self-reciprocal polynomial in $\R[x]$ of degree $2n$.  Assume   $H(P)\le H$. Let $P(\mu)=0$, $\vert \omega\vert=1$, $\vert P(\omega)\vert\ge 1$. Then 
\[\vert\mu-\omega\vert \gtrsim_{n\rightarrow \infty} H^{-1} n^{-2}.\]
Besides, if  $\vert \mu\vert\not=1$ then  
\[\vert \mu\vert^{-\frac{1}{2}}\vert\mu-\omega\vert\gtrsim_{n\rightarrow \infty} \left(2c^{-1}H^{-1}\right)^{\frac{1}{2}} n^{-\frac{3}{2}},\]
where $c>1$ is such that $c+1=c \log c\ \ (c=3.594...)$.

\noindent For $\omega\in \{-1,1\}$:
\[\vert \mu\vert^{-\frac{1}{2}}\vert\mu-\omega\vert \gtrsim_{n\rightarrow \infty}\left(3H^{-1}\right)^{\frac{1}{2}} n^{-\frac{3}{2}}.\]
If $\vert \mu\vert\not=1$  and $\mu \notin \R$, then  
\[\vert \mu\vert^{-\frac{1}{2}}\vert\mu-\omega\vert \gtrsim_{n\rightarrow \infty} (2e^{-1})^{\frac{1}{2}}\left(3H^{-1}\right)^{\frac{1}{4}}n^{-\frac{5}{4}}.\]

\end{cor}

\begin{proof}
We observe that
\[\sum_{j=0}^{n-1}(n-j)\vert a_j\vert\le  \frac{n(n+1)}{2}H,\]
\[\sum_{j=0}^{n-1}(n-j)^2\vert a_j\vert\le \frac{n(n+1)(2n+1)}{6}H.\]
Then we apply Theorem \ref{lower1} with different values of  $\delta$ in order to get the best upper bound:
$\delta=1+\frac{1}{\sqrt{n}}$ in (\ref{alpha}), $c$  in (\ref{betagamma}),  $\delta=1+\frac{1}{n^{\frac{1}{4}}}$ in (\ref{alphabeta}) and finally $\delta=e^2$ in (\ref{gamma}).
\end{proof}

\begin{cor}\label{lower4}
Let $P(x)=\displaystyle\sum_{j=0}^{2n}a_jx^j$ be a self-reciprocal polynomial in $\R[x]$ of degree $2n$.  Assume $a_j\ge 0$ $(j=0,\ldots, n)$. Let $P(\mu)=0$. Then
\[\vert \mu\vert^{-\frac{1}{2}}\vert\mu-1\vert \gtrsim_{n\rightarrow \infty}\frac{A}{n}.\]
Besides, if $\vert \mu\vert\not=1$  and $\mu\notin \R $, then 
\[\vert \mu\vert^{-\frac{1}{2}}\vert\mu-1\vert \gtrsim_{n\rightarrow \infty} \frac{B}{n},\]
where $A$ and $B$ are given by
\[A^2= 4\left(a(2+\log a)\right)^{-1}, a>1, a(\log a)^3=4 \ \ (A=0.655...).\]
 \[B^4=8\log b^2\left(b(\log b+2)\right)^{-1}, b>1,  b(\log b)^2(\log b-2)=8\ \  (B=0.984...).\]
\end{cor}

\begin{proof} 
\[\sum_{j=0}^{n-1}(n-j)^2\vert a_j\vert\le \frac{n^2}{2} L(P)\]
and $L(P)=P(1)=\Vert P\Vert$.
Then we apply Theorem \ref{lower1} with 
$\delta=a$  in (\ref{alphabeta}) and  with $\delta=b$ in (\ref{gamma}).
\end{proof}

\begin{cor}\label{lower3}
Let $P(x)=\displaystyle\sum_{j=0}^{2n}a_jx^j$ be a self-reciprocal polynomial in $\R[x]$ of degree $2n$. Let $P(\mu)=0$, $\vert \omega\vert=1$, $\Vert P\Vert=\vert P(\omega)\vert$. Then
\[\vert\mu-\omega\vert \gtrsim_{n\rightarrow \infty} 3^{\frac{1}{2}}2^{-\frac{1}{2}}n^{-\frac{3}{2}}.\]
Besides, if $\vert \mu\vert\not=1$, then
\[\vert \mu\vert^{-\frac{1}{2}}\vert\mu-\omega\vert \gtrsim_{n\rightarrow \infty} 6^{\frac{1}{4} }c^{-\frac{1}{2}}n^{-\frac{5}{4}},\]
where $c>1$ is such that $c+1=c \log c\ \ (c=3.594...)$.
\noindent For $\omega\in \{-1,1\}$:
\[\vert \mu\vert^{-\frac{1}{2}}\vert\mu-\omega\vert \gtrsim_{n\rightarrow \infty} 10^{\frac{1}{4}} n^{-\frac{5}{4}}.\]
If  $\vert \mu\vert\not=1$  and $\mu\notin \R$, then
\[\vert \mu\vert^{-\frac{1}{2}}\vert\mu-\omega\vert \gtrsim_{n\rightarrow \infty} \left(2e^{-1}\right)^{\frac{1}{2}}10^\frac{1}{8} n^{-\frac{9}{8}}.\]

\end{cor}
\begin{proof} By Cauchy-Schwarz inequality,
\[\sum_{j=0}^{n-1}(n-j)\vert a_j\vert\le \frac{1}{\sqrt 12} \left(n(n+1)(2n+1)\right)^{\frac{1}{2}} L_2(P),\]
\[\sum_{j=0}^{n-1}(n-j)^2\vert a_j\vert\le \frac{1}{\sqrt{60}} \left(n(n+1)(2n+1)(3n^2+3n-1)\right)^{\frac{1}{2}}L_2(P)\]
and $L_2(P)\le \Vert P\Vert=\vert P(\omega)\vert$. 
Then we apply Theorem \ref{lower1} with 
$\delta=1+\frac{1}{n^{\frac{1}{4}}}$ in (\ref{alpha}), $c>1$ in (\ref{betagamma}), $\delta=1+\frac{1}{n^{\frac{1}{8}}}$ in (\ref{alphabeta}) and finally $\delta=e^2$ in (\ref{gamma}).
\end{proof}
\section{A lower bound for the Mahler measure of complex polynomials}
A result due to Schinzel \cite{Schinzel 73} states that any monic polynomial $P\in \Z[x]$ of degree $d$  whose all zeros are real (resp.  positive) and such that $P(-1)P(1)\not=0$ and $\vert P(0)\vert=1$ satisfies 
\[M(P)\ge \left(\frac{1+\sqrt 5}{2}\right)^{\frac{d}{2}}\  (\text{resp. } M(P)\ge \left(\frac{1+\sqrt 5}{2}\right)^d).\] 
V. Flammang \cite{Flammang 97} generalised Schinzel's theorem  by proving that any monic $P\in \R[x]$ of degree $d$  whose all zeros are real (resp. positive) and such that $P(0)\not=0$, $\vert P(1)\vert\ge 1$,  $\vert P(-1)\vert\ge 1$ satisfies 
\[M(P)\ge \left(\frac{1+\left(4\vert P(0)\vert^{\frac{2}{d}}+1\right)^{\frac{1}{2}}}{2}\right)^{\frac{d}{2}}\ (\text{resp. } M(P)\ge \left(\frac{1+\left(4\vert P(0)\vert^{\frac{1}{d}}+1\right)^{\frac{1}{2}}}{2}\right)^d).\] 
J. Garza \cite{Garza} showed that for monic polynomials $P\in \Z[x]$ satisfying  $P(0)P(-1)P(1)\not=0$ and having $m\ge 1$ real zeros, the following holds:
\[M(P)\ge \left(\frac{1+\left(4^{\frac{d}{m}}+1\right)^{\frac{1}{2}} }{2^{\frac{d}{m}}}\right)^{\frac{m}{2}}.\]
Our next theorem  generalises the three above results to polynomials in $\C[x]$. It also strengthens the dependence on $P(0)$, $P(1)$ and $P(-1)$. 
\begin{thm}\label{Schinzel} Let $P\in \C[x]$ be a  monic polynomial of degree $d\ge 2$.  
Assume that $P$ has $m\ge 1$ real zeros and that  $P(0)P(1)P(-1)\not=0$, then 
\begin{equation}\label{m}
M(P)\ge \left(\frac{\vert P(1)P(-1)\vert^{\frac{1}{m}}+\left(4^{\frac{d}{m}}\vert P(0)\vert^{\frac{2}{m}}+\vert P(1)P(-1)\vert^{\frac{2}{m}}\right)^{\frac{1}{2}}}{2^{\frac{d}{m}}}\right)^{\frac{m}{2}}
\end{equation}
with equality if and only if $P$ is of the form
\begin{equation}\label{equality}
P=(z-i)^{d_1}(z+i)^{d_2}(z-a)^{d_3}(z+a)^{d_4}(z-a^{-1})^{d_5}(z+a^{-1})^{d_6},\ \  a>1.
\end{equation}
 If $P$ has $n\ge 1$ positive zeros  and $P(0)P(1)\not=0$, then
\begin{equation}\label{n}
M(P)\ge\left(\frac{\vert P(1)\vert^{\frac{1}{n}}+\left(4^{\frac{d}{n}}\vert P(0)\vert^{\frac{1}{n}}+\vert P(1)\vert^{\frac{2}{n}}\right)^{\frac{1}{2}}}{2^{\frac{d}{n}}}\right)^{n}
\end{equation}
with equality if and only if $P$ is of the form
\[P=(z-a)^{d_1}(z-a^{-1})^{d_2},\ \  a>1.\]
\end{thm}
\begin{proof} Let $Q=P(0)^{-1}PP^*$ defined by (\ref{Q}). Let $\{\beta_k\}_{k=1}^m$ be the real zeros of $Q$ outside $[-1,1]$  and $\{\delta_k\}_{k=1}^f$  its non-real zeros such that either $\vert \delta_k\vert> 1$ or $\vert \delta_k\vert= 1$, $\Im \delta_k>0$. Then $d=m+f$ and 
 \begin{equation*}\label{P(1)P(-1)}
\begin{split}
Q(1)&=\prod_{k=1}^f (1- \delta_k)(1-\delta_k^{-1})\prod_{k=1}^m (1- \beta_k)(1-\beta_k^{-1})\\
Q(-1)&=\prod_{k=1}^f (1+ \delta_k)(1+\delta_k^{-1})\prod_{k=1}^m (1+ \beta_k)(1+\beta_k^{-1})\\
\vert Q(1)Q(-1)\vert^{\frac{1}{2}}&=M(Q)\prod_{k=1}^f\vert 1-\delta_k^{-2}\vert\prod_{k=1}^m (1-\beta_k^{-2}) \\
\end{split}
\end{equation*}
Like in the proof of Schinzel's theorem, we use the inequality \cite{Schinzel 73}[Lemma 3]
\[\prod_{k=1}^m (1-\beta_k^{-2})\le \left(1-\left(\prod_{k=1}^m \beta_k^{-2}\right)^{\frac{1}{m}}\right)^m\le\left(1-M(Q)^{-\frac{2}{m}}\right)^m,\]
which follows from the concavity of the function $t\mapsto \log(1-e^{-t}), t>0$. We obtain
\begin{equation}\label{M}
\begin{split}
\vert Q(1)Q(-1)\vert^{\frac{1}{2}}\le M(Q)2^{d-m}\left(1-M(Q)^{-\frac{2}{m}}\right)^m.
\end{split}
\end{equation}
We deduce that
\[M(Q)^{\frac{2}{m}}-2^{1-\frac{d}{m}}\vert Q(1)Q(-1)\vert^{\frac{1}{2m}} M(Q)^{\frac{1}{m}}-1\ge 0\] and solving the inequality $t^2-2^{1-\frac{d}{m}}\vert Q(1)Q(-1)\vert^{\frac{1}{2m}} t-1\ge 0$, we find
\begin{equation*}\label{pf}
2^{\frac{d}{m}}M(Q)^{\frac{1}{m}}\ge \vert Q(1)Q(-1)\vert^{\frac{1}{2m}}+\left(4^{\frac{d}{m}}+\vert Q(1)Q(-1)\vert^{\frac{1}{m}}\right)^{\frac{1}{2}}.
\end{equation*}
We conclude the proof of (\ref{m}) by using (\ref{Q-P}). 

Equality in (\ref{m}) is equivalent to equality in (\ref{M}). It is attained iff 
\[\prod_{k=1}^f \frac{2}{\vert 1-\delta_k^{-2}\vert}=\left(1-M(Q)^{-\frac{2}{m}}\right)^{-m}\prod_{k=1}^m (1-\beta_k^{-2})=1,\]
which gives
\[\delta_1^2=\ldots=\delta_f^2=-1,\ \ \ \beta_1^2=\ldots=\beta_m^2=a^2 \text{ with } a>1.\]

\noindent For the proof of (\ref{n}), let $\{\beta_k\}_{k=1}^n$ be the positive zeros of $Q$ and $\{\beta_k\}_{k=n+1}^m$ be the negative zeros of $Q$. Using once again the concavity of the function $t\mapsto \log(1-e^{-t}), t>0$, we have 
\begin{equation}\label{N}
\begin{split}
\vert Q(1)\vert^{\frac{1}{2}}&=M(Q)^{\frac{1}{2}}\prod_{k=1}^f \vert 1- \delta_k^{-1}\vert\prod_{k=1}^n(1-\beta_k^{-1})\prod_{k=n+1}^m (1-\beta_k^{-1})\\
&\le M(Q)^{\frac{1}{2}}2^{d-n} \prod_{k=1}^n(1-\beta_k^{-1})\\
&\le M(Q)^{\frac{1}{2}}2^{d-n} \left(1-M(Q)^{-\frac{1}{n}}\right)^n.
\end{split}
\end{equation}
We find 
\[M(Q)^{\frac{1}{n}}-2^{1-\frac{d}{n}}\vert Q(1)\vert^{\frac{1}{2n}} M(Q)^{\frac{1}{2n}}-1\ge 0.\] 
We proceed as in the end of the proof of (\ref{m}) by substituting  $M(Q)^{\frac{1}{2n}}$ to $M(Q)^{\frac{1}{m}}$ and $\vert Q(1)\vert^{{\frac{1}{2n}}}$ to $\vert Q(1)Q(-1)\vert^{\frac{1}{2m}}$.

Equality in (\ref{n}) is equivalent to equalities in (\ref{N}). It is attained iff
\[\prod_{k=1}^f \frac{2}{\vert 1-\delta_k^{-1}\vert}\prod_{k=n+1}^{m}\frac{2}{\vert 1-\beta_k^{-1}\vert}=\left(1-M(Q)^{-\frac{1}{n}}\right)^{-n}\prod_{k=1}^m (1-\beta_k^{-1})=1,\]
which gives
\[f=0,\ \ \ n=m=d,\ \ \  \beta_1=\ldots=\beta_n=a \text{ with } a>1.\]

\end{proof}

\section{The number of real zeros}

\subsection{Bounds involving the Mahler measure}
A. Dubickas \cite{Dubickas 95} showed that for polynomials $P\in \Z[x]$ of degree $d$ such that $P(0)P(-1)P(1)\not=0$, there is a constant $D_0$ such that for all $d\ge D_0$,
\begin{equation}\label{Dubikas}
\vert P^{-1}(0)\cap \R\vert\le 1.085\sqrt{d\log d \log M(P)}.
\end{equation}
A bound of this size does not apply to polynomials in $\R[x]$. For example, the polynomials of the form $P=\left(x-d^{\frac{1}{d}}\right)^{d}$ satisfy $M(P)=d$,
\[\sqrt{d\log d \log M(P)}=\sqrt {d} \log d \text{ \ \ while }  \ \ \vert P^{-1}(0)\cap \R\vert=d=\frac{d\log M(P)}{\log d}\ .\]
Our next theorem shows that  the bound $\displaystyle \frac{d\log (2 M(P))}{\log d}$ holds, up to a small multiplicative constant,  for polynomials in $P\in \C[x]$ such that $ \vert P(1)P(-1)\vert \ge 1$ and $ \vert P(0)\vert \ge 1$. 
\begin{thm}\label{com} Let $P\in \C[x]$ be a monic  polynomial of degree $d\ge 2$. Assume  $P(0)P(1)P(-1)\not=0$. Then
\begin{equation}\label{COMPLEX}
\begin{split}
\vert P^{-1}(0)\cap \R\vert& \le \max\left( \frac{d\log 2+\log \vert P(0)\vert-\log \left\vert P(1)P(-1)\right\vert}{-\log \left(\sinh\left(\frac{\log d}{d}\right )\right)}, \frac{d\log \left( M(P)^2\vert P(0)\vert^{-1}\right)}{\log d}\right). \\
\end{split}
\end{equation}
\end{thm}

We note that $\displaystyle -\log \left(\sinh\left(\frac{\log d}{d}\right )\right)\sim \log d$.
\begin{rem}
The equality in (\ref{COMPLEX}) is attained by polynomials of the form  (\ref{equality}) with $a=d^{\frac{1}{d}}$.
If we denote by $m=\vert P^{-1}(0)\cap \R\vert$, they satisfy 
\[M(P)^2\vert P(0)\vert^{-1}=d^{\frac{m}{d}}, \ \ \vert P(0)\vert^{-1}\vert P(1)P(-1)\vert=2^d\left(\sinh \left(\frac{\log d}{d}\right )\right)^m.\]
Thus 
\[\frac{d\log 2+\log \vert P(0)\vert-\log \left\vert P(1)P(-1)\right\vert}{-\log \left(\sinh\left(\frac{\log d}{d}\right )\right)}=\frac{d\log \left( M(P)^2\vert P(0)\vert^{-1}\right)}{\log d}=m.\]

\end{rem}

\begin{proof}[Proof of Theorem \ref{com}]${}$

Let us repeat the beginning  of the proof of Theorem \ref{Schinzel} with the same notations. We assume that  $\displaystyle m\ge  \frac{d\log M(Q)}{\log d}$
 and we use inequality (\ref{M}):
 \begin{equation*}
\begin{split}
\vert Q(1) Q(-1)\vert^{\frac{1}{2}}&\le M(Q)2^{d-m}\left(1-M(Q)^{-\frac{2}{m}}\right)^m\\
&=2^{d}\left(\sinh\left(\frac{\log M(Q)}{m}\right)\right)^m\\
&\le  2^{d}\left(\sinh\left(\frac{\log d}{d}\right)\right)^m\\
\end{split}
\end{equation*}
We deduce that
\[m \le \frac{d\log 2-\frac{1}{2}\log\left\vert Q(1)Q(-1)\right\vert}{-\log \left(\sinh\left(\frac{\log d}{d}\right )\right)}.\]
The result follows by using (\ref{Q-P}).

\end{proof}

\subsection{Bounds involving  the Mahler measure and the length}
A  theorem announced by  E. Schmidt in 1932 states that for  polynomials of the form  $P(x)=\displaystyle\sum_{j=0}^d a_j x^j$, $a_0a_d\not=0$, $a_j\in \C$, there  is an absolute constant $C>0$ such that 
\begin{equation}\label{Schmidt}
\vert P^{-1}(0)\cap \R\vert^2\le C\,d\log\left(\frac{L(P)}{\sqrt{\vert a_0a_d\vert}}\right). 
\end{equation}
  I. Schur \cite{Schur 33} gave an elementary proof of (\ref{Schmidt}) and  showed that the  best possible constant  $C$ is $4$.
  
  P. Borwein, T. Erd\'elyi and G. K\'os  \cite {Bornwein Erdelyi and Kos 99} showed that whenever $\vert a_0\vert=\vert a_d\vert=1$, $\vert a_j\vert\le 1$,  there is an absolute constant $C>0$ such that 
\begin{equation*}\label{BEK}
\vert P^{-1}(0)\cap \R\vert\le C\sqrt d. 
\end{equation*}

Our next result gives an upper bound depending both on the Mahler measure and on the length of the polynomial.
 \begin{thm}\label{L(P)} Let $P\in \C[x]$ be a monic polynomial of degree $d$.    Assume $P(0)P(1)P(-1)\not=0$. Then
\begin{equation}\label{complex}
\begin{split}
\vert P^{-1}(0)\cap \R\vert&\le 2\sqrt c 
\left(d\log \frac{M(P)^2}{\vert P(0)\vert}\right)^{\frac{1}{2}}+(c+1)\log \frac{M(P)^2}{\vert P(0)\vert}\\
&+(\log c)^{-1}\log(1+d^2) +(\log c)^{-1}\log\left(\frac{L(P)^2}{2\vert P(1)\vert^2}+\frac{L(P)^2}{2\vert P(-1)\vert^2}\right),\\
\end{split}
\end{equation}
where $c>1$ is  such that $c\log c=1+c\ \ (c=3.594...)$.
\end{thm}
\begin{proof}
Once again we use $Q=P(0)^{-1}PP^*$ and we may assume that $M(Q)>1$ otherwise there are no real zeros at all.
We apply Lemma \ref{Jensen} to $Q$ with $\rho=\frac{\gamma}{\gamma+1}$, $\gamma=\sqrt{\frac{(c+1)\log c\log M(Q)}{d}}$ and first with $\omega=1$: 
\begin{equation}\label{1}
\begin{split}
2\log c\,\vert Q^{-1}(0)\cap \{1<x \le 1+\frac{\rho}{c}\} \vert&\le \sum_{\underset{1< \mu \le 1+\frac{\rho}{c}}{Q(\mu)=0}} \left(\log \frac{\rho}{\vert \mu-1\vert}+\log \frac{\rho}{\vert \mu^{-1}-1\vert}\right)\\
&\le \sum_{\underset{ \vert \mu-1\vert\le \rho}{Q(\mu)=0}} \log \frac{\rho}{\vert \mu-1\vert}\\
&\le  \log \left(1+\frac{\rho^2(1-\rho)^{-d}d^2 }{ \vert Q(1)\vert}\sum_{k=0}^{d-1} \vert A_k\vert\right).
\end{split}
\end{equation}
Repeating the argument with $\omega=-1$,
\begin{equation}\label{2}
2\log c\,\vert Q^{-1}(0)\cap \{-1-\frac{\rho}{c}\le x < -1\}\vert \le \log \left(1+\frac{\rho^2(1-\rho)^{-d}d^2 }{\vert Q(-1)\vert}\sum_{k=0}^{d-1} \vert A_k\vert\right).
\end{equation}
Adding (\ref{1}) and (\ref{2}) and using the inequality $(1+u)(1+v)\le \left(1+\frac{u+v}{2}\right)^2$, 
\begin{equation*}
\begin{split}
\log c\,\vert Q^{-1}(0)\cap \{1<\vert x \vert \le 1+\frac{\rho}{c}\}\vert 
&\le  \log\left(1+d^2 \rho^2(1-\rho)^{-d}\left(\frac{1}{2\vert Q(1)\vert}+\frac{1}{2\vert Q(-1)\vert}\right)\sum_{k=0}^{d-1} \vert A_k\vert\right) \\
&\le \log\left(\left(\frac{1}{2\vert Q(1)\vert}+\frac{1}{2\vert Q(-1)\vert}\right)\sum_{k=0}^{d} \vert A_k\vert \right)+\log(1+d^2\rho^2)-d\log(1-\rho)\\
\end{split}
\end{equation*}
Now we use the inequality $-\log(1-\rho)\le \frac{\rho}{1-\rho}=\gamma$,  (\ref{Q-P}) and (\ref{A-P}) to obtain
\[\log c\vert Q^{-1}(0)\cap \{1<\vert x\vert\le 1+\frac{\rho}{c}\}\vert \le  \log\left(\frac{L(P)^2}{2\vert P(1)\vert^2}+\frac{L(P)^2}{2\vert P(-1)\vert^2}\right)+\log(1+d^2)+\gamma d.\]
On the other hand, by Proposition \ref{b_r} and  the inequality $\log(1+\frac{\rho}{c})\ge \frac{\rho}{c+1}$,
\[\vert Q^{-1}(0)\cap \{\vert x\vert>1+\frac{\rho}{c}\}\vert\le  \frac{\log M(Q)}{\log(1+\frac{\rho}{c})}\le \log M(Q)(c+1)\rho^{-1}=\log M(Q)(c+1)(1+\gamma^{-1})\]
\[\le(c+1) \log M(Q)+\sqrt{c\log M(Q)d}.\]
Finally,
\begin{equation*}
\begin{split}
\vert P^{-1}(0)\cap \R \vert&= \vert Q^{-1}(0)\cap \{\vert x\vert>1\}\vert\\
& \le (\log c)^{-1}\log\left(\frac{L(P)^2}{2\vert P(1)\vert^2}+ \frac{L(P)^2}{2\vert P(-1)\vert^2}\right)+\\
&+(\log c)^{-1}\log(1+d^2)+(c+1)\log M(Q)+2\sqrt{c\log M(Q)d}.\\
\end{split}
\end{equation*}
We conclude the proof of (\ref{complex}) by applying (\ref{Q-P}).
\end{proof}
\begin{cor}\label{L(P)Z} 
 Let $P\in \Z[x]$ be of degree $d$.    Assume $P(0)P(1)P(-1)\not=0$. Then 
 \begin{equation}\label{integer}
\begin{split}\vert P^{-1}(0)\cap \R\vert&\le 2\sqrt{2c\,d\log M(P)}+2(c+1)\log M(P)\\
&+(\log c)^{-1}\log(1+d^2)+2(\log c)^{-1}\log L(P),
\end{split}
\end{equation}
where $c>1$ is  such that $c\log c=1+c\ \ (c=3.594...)$.

\noindent If moreover, $\log L(P)\le c_1\sqrt{d\log M(P)}$ for some constant $c_1>0$,  then for all $\varepsilon>0$, there exists $D_0(\varepsilon)$ such that for $d\ge D_0(\varepsilon)$, 
\begin{equation}\label{integer1}
\vert P^{-1}(0)\cap \R\vert\le (c_2+\varepsilon) \sqrt{d\log M(P)},
\end{equation}
where  $c_2=2\left(\sqrt{2c}+c+1+(\log c)^{-1} c_1\right)$.
\end{cor}
\begin{rem}
Inequality (\ref{integer1}) improves the bound  (\ref{Dubikas}) given by Dubickas  in the particular case where $L(P)\le c_1 \sqrt{d\log M(P)}$.
\end{rem}
\begin{proof}
Let $P=a_dx^d+\ldots+a_0$ with $a_0,\ldots,a_d\in \Z$, $a_d\not=0$ and $P(0)P(1)P(-1)\not=0$. We apply Theorem \ref{L(P)} to the polynomial $\tilde P=a_d^{-1}P$. Inequality (\ref{integer}) follows after observing that
\[
\begin{split}
\frac{M(\tilde P)^2}{\vert \tilde P(0)\vert}&=\frac{M(P)^2}{a_d\vert \tilde P(0)\vert}\le M(P)^2\\
\frac{L(\tilde P)^2}{2\vert \tilde P(1)\vert^2}+\frac{L(\tilde P)^2}{2\vert \tilde P(-1)\vert^2}&=\frac{L(P)^2}{2\vert P(1)\vert^2}+\frac{L(P)^2}{2\vert P(-1)\vert^2}\le L(P)^2.
\end{split}
\] 
 
 For the proof of  (\ref{integer1}), we may  assume that  $M(P)>1$,  otherwise by Kronecker's  theorem $P$ has no real zeros. Then,
 we know from \cite{Dobrowolski 79} that for some constant $\sigma>0$,
\begin{equation}\label{dobr}
\log M(P)\ge \sigma\left(\frac{\log\log d}{\log d}\right)^3.
\end{equation}
Thus for all $\varepsilon>0$, there exists $D_0(\varepsilon)$ such that for $d\ge D_0(\varepsilon)$, 
\[(\log c)^{-1}\log(1+d^2)\le \varepsilon \sqrt{d\log M(P)}.\]
On the other hand, we may also assume that $\log M(p)\le d$  because otherwise $\sqrt{d\log M(P)}\ge d$ and  (\ref{integer1})   is  again trivially satisfied. 

\noindent  Since $\log M(p)\le \sqrt{d\log M(p)}$,   (\ref{integer1})  follows immediately from (\ref{integer}) and the assumption that $\log L(P)\le c_1\sqrt{d\log M(P)}$.  
\end{proof}


\section{The number of zeros in $\{\vert x-1\vert< 1\}$.}
In \cite{Dubickas 95}, A. Dubickas used a Vandermonde determinant which, in its simplified form is defined by 
\[d(x)=\det(x^{(j-1)l})_{1\le j,l\le N}=\prod_{1\le v<u \le N} (x^u-x^v)\ \ \ (N\in \N^*).\]
  Following the lines of his work, we put
\[R(x)=\vert d(x)\vert\prod_{1\le v<u \le N} \vert x\vert^{-v}=\prod_{j=1}^{N-1}\vert x^{j}-1\vert^{N-j}.\]
In \cite{Dubickas 95}, the author applies Hadamard's inequality to obtain:
\begin{equation}\label{R}
R(x)\le  \max(1,\vert x\vert)^{\frac{(N-1)N(N+1)}{6}}N^{\frac{N}{2}}.
\end{equation}

\begin{lem}\label{K}
Assume $P\in \Z[x]$ is monic, has degree $d$ and  does not  vanish at the roots of unity. Then for all integer $N\ge 2$,
\[\vert P(1)\vert\le N^{\frac{d}{N-1}} M(P)^{\frac{N+1}{3}}.\]
\end{lem}
\begin{rem}${}$
\begin{itemize}
\item By Kronecker's theorem, either $P=x^d$ or $M(P)>1$. 
\item By (\ref{ineq}) and (\ref{Parseval}), we already had  $\vert P(1)\vert\le 2^d M(P)$.   We retrieve this  inequality if we put $N=2$ in Lemma \ref{K}.
\end{itemize}
\end{rem}
\begin{proof}
Let  $\mu_1,\ldots,\mu_d$ be the zeros of $P$. Then
\begin{equation*}\label{R2}
\begin{split}
R(\mu_k)&=\vert \mu_k-1\vert^{\frac{(N-1)N}{2}}\prod_{j=1}^{N-1}\left\vert\sum_{l=0}^{j-1}\mu_k^l\right\vert^{N-j}\\
\end{split}
\end{equation*}
and
\[\prod_{k=1}^d R(\mu_k)=\vert P(1)\vert^{\frac{(N-1)N}{2}}\prod_{k=1}^d \prod_{j=1}^{N-1}\left\vert\sum_{l=0}^{j-1}\mu_k^l\right\vert^{N-j}.\]
Observing that 
\[\prod_{k=1}^d \prod_{j=1}^{N-1}\left(\sum_{l=0}^{j-1}\mu_k^l\right)^{N-j}\in \Z^*,\]
and using the inequality (\ref{R}), we deduce that 
\begin{equation*}
\begin{split}
\vert P(1)\vert^{\frac{(N-1)N}{2}}&\le\prod_{k=1}^d R(\mu_k)\le N^{\frac{dN}{2}} M(P)^{\frac{(N-1)N(N+1)}{6}}.
\end{split}
\end{equation*}
The result follows easily.
\end{proof}
We will need the following  result due  to  Zhang \cite{Zhang 92} and for which Zagier gave an elementary proof in \cite{Zagier 93}:
\begin{thm}\label{Zhang}[Zhang, Zagier]Let $\omega$ be a primitive $6$th root of unity. Let $P\in \Z[x]$  be a polynomial of degree $d$ such that $P(0)P(1)P(\omega)\not=0$. 
Denote $P_*(x)=P(1-x)$. Then
\[M(P)M(P_*)\ge \left(\frac{1+\sqrt 5}{2}\right)^{\frac{d}{2}}.\]
\end{thm}
\begin{thm}\label{around1}
Let  $P\in \Z[x]$ be monic, irreducible, of degree $d$  and   satisfying  $0<\log M(P)\le \phi(d)$ where $\phi(d)=o(\frac{d}{\log d})$. Then for 
all $\varepsilon>0$, there exists $D_0(\varepsilon)$ such that for $d\ge D_0(\varepsilon)$, 
\[\vert P^{-1}(0)\cap\{\vert x-1\vert<1\}\vert\ge \left(\frac{2}{\pi}\log \frac{1+\sqrt 5}{2}-\varepsilon\right)\sqrt{\frac{d}{\log d \log M(P)}}.\]
This leads to the lower bound 
\begin{equation}\label{minor}
\log M(P) \ge   \left(\frac{2}{\pi}\log\frac{1+\sqrt 5}{2}-\varepsilon\right)^2\frac{d}{K(P)^2\log d },
\end{equation}
where $K(P)=\min\left(\left\vert P^{-1}(0)\cap\{\vert x-1\vert<1\}\right\vert, \left\vert P^{-1}(0)\cap\{\vert x+1\vert<1\}\right\vert \right)$.

\end{thm}

\begin{proof}
We apply Lemma \ref{K} with 
$\displaystyle N=\left[\sqrt\frac{d\log d }{\log M(P)}\right]+2.$ We obtain
\begin{equation*}
\begin{split}
\frac{N+1}{3}\log M(P)&\le \log M(P)+\frac{1}{3}\sqrt{d\log d \log M(P)},\\
d\frac{\log N}{N-1}&\le \sqrt \frac{d\log M(P)}{\log d} \log\left(2\sqrt\frac{d\log d }{\log M(P)}\right)\\
&=\sqrt{d\log d \log M(P)}\left(\frac{1}{2}+\frac{\log (4\log d)}{2\log d}\right)-\sqrt{\frac{d}{\log d}}\sqrt{\log M(P)}\log\sqrt{\log M(P)}\\
&\le \sqrt{d\log d \log M(P)}\left(\frac{1}{2}+\frac{\log (4\log d)}{2\log d}\right)+e^{-1}\sqrt{\frac{d}{\log d}}. \\
\end{split}
\end{equation*}
\[\vert P(1)\vert\le M(P)\exp\left(\sqrt{d\log d \log M(P)}\left(\frac{5}{6}+\frac{\log (4\log d)}{2\log d}\right)+\sqrt{\frac{d}{\log d}}\right).\]
Denote by $\mu_1,\ldots,\mu_d$ the zeros of $P$.
By Theorem \ref{Zhang}, we have
\[\prod_{\vert 1-\mu_k\vert\ge 1}\vert 1-\mu_k\vert\ge \left(\frac{1+\sqrt 5}{2}\right)^{\frac{d}{2}}M(P)^{-1}.\]
We immediately deduce that 
\begin{equation}\label{P(1)}
\begin{split}
\prod_{\vert 1-\mu_k\vert< 1}\vert 1-\mu_k\vert&\le \vert P(1)\vert\left(\frac{1+\sqrt 5}{2}\right)^{-\frac{d}{2}}M(P)\\
&\le M(P)^{2}\exp\left({\sqrt{d\log d \log M(P)}}\left(\frac{5}{6}+\frac{\log (4\log d)}{2\log d}\right)+\sqrt{\frac{d}{\log d}}-\frac{d}{2}\log\left(\frac{1+\sqrt 5}{2}\right)\right)\\
\end{split}
\end{equation}
Denote by $J=\vert P^{-1}(0)\cap\{\vert x-1\vert<1\}\vert$ and let $\varepsilon>0$. Using the lower bound given by A. Dubickas (see \cite{Dubickas 95} or (\ref{Dub})),  there exists $D_0(\varepsilon)$ such that for all $d\ge D_0(\varepsilon)$ 
\begin{equation}\label{dub}
\prod_{\vert 1-\mu_k\vert\le 1}\vert 1-\mu_k\vert\ge  \exp\left(-J( \frac{\pi}{4}+\varepsilon')\sqrt{d\log d \log M(P)}\right)
\end{equation}
where $\varepsilon'>0$ is chosen small enough with regards to $\varepsilon$.
Using   (\ref{P(1)}), (\ref{dub}) and $\log M(P)\le  \phi(d)=o(\frac{d}{\log d})$, we find  for $d\ge D_0(\varepsilon)$
\[J\ge \left(\frac{2}{\pi}\log \frac{1+\sqrt 5}{2}-\varepsilon\right)\sqrt{\frac{d}{\log d \log M(P)}}.\]
As the same apply to $P(-x)$, we  have the same estimate for $\vert P^{-1}(0)\cap\{\vert x+1\vert<1\}\vert$ and we obtain (\ref{minor}).
\end{proof}

\section{Acknowledgments}
We sincerely thank the anonymous reviewer, whose comments and suggestions helped us to improve and clarify our manuscript.

\end{document}